\documentclass[a4paper,11pt,twoside,DIV=calc]{scrartcl}	

\usepackage{tocenter}				
\ToCenter[h]{\textwidth}{\textheight}	

\setkomafont{section}{\bfseries\Large}		

\usepackage{scrpage2}      	
\ofoot{}					
\ifoot{}
\rehead{\pagemark}			
\cehead{OLIVIER GABRIEL AND MARTIN GRENSING}	
\rohead{\pagemark}			
\cohead{SMOOTH SUBALGEBRAS OF PIMSNER ALGEBRAS}

\input xy
\xyoption{all}
\usepackage{lscape}
\usepackage[latin1]{inputenc}
\usepackage[cmtip,arrow]{xy}
\usepackage[leqno]{amsmath}
\usepackage{amssymb,amscd,amsthm,mathrsfs,yfonts,manfnt,pb-diagram,pb-xy}
\usepackage[nice]{nicefrac}		
\usepackage[T1]{fontenc}

\usepackage{ifthen}
\newboolean{comm}
\setboolean{comm}{true}
\newboolean{sol}
\setboolean{sol}{false}
\newboolean{notes}
\setboolean{notes}{true}			
\CompileMatrices  	

\newcommand{\com}{\ifthenelse{\boolean{comm}}}
\newcommand{\sol}{\ifthenelse{\boolean{sol}}}
\newcommand{\note}{\ifthenelse{\boolean{notes}}}

\newtheorem{Def}{Definition}
\newtheorem{Prop}[Def]{Proposition}

\newtheorem{Th}[Def]{Theorem}

\newtheorem{Lem}[Def]{Lemma}
\newtheorem{Cor}[Def]{Corollary}

\theoremstyle{definition}
\newtheorem{Rem}[Def]{Remark}

\newcommand{\mC}{\ensuremath{\mathbb{C}}}					
\newcommand{\mN}{\ensuremath{\mathbb{N}}}
\newcommand{\mZ}{\ensuremath{\mathbb{Z}}}
\newcommand{\mK}{\ensuremath{\mathbb{K}}}
\newcommand{\mT}{\ensuremath{\mathbb{T}}}

\newcommand{\mB}{\ensuremath{\mathbb{B}}}

\newcommand{\mc}{\mathcal}								

\DeclareMathOperator{\Image}{Im}

\newcommand{\ee}{\mbox{$\varepsilon$}}

\renewcommand{\phi}{\varphi}

\newcommand{\sph}{\mathbb{S}}							

\newcommand{\fn}{\mathcal{C}}	   				  

\newcommand{\mb}{\begin{pmatrix}}					
\newcommand{\me}{\end{pmatrix}}						
\newcommand{\co}{\mc{K}(\mc{H})}						
\newcommand{\toe}{\mathcal{T}}						
\newcommand{\stoe}{\ensuremath{\mathcal{T}^\infty}}		

\DeclareMathOperator{\Ad}{Ad}							
\newcommand{\sco}{\ensuremath{\mathcal{K}}}					

\DeclareMathOperator{\Ker}{Ker}						

\DeclareMathOperator{\id}{id}							

\newcommand{\rrarrow}{\rightrightarrows}

\usepackage[T1]{fontenc}
\usepackage{lmodern}
\usepackage{graphicx}
\usepackage[english]{babel}
\usepackage{amsfonts}
\usepackage{amssymb,amsmath,amsthm,amscd}
\usepackage{mathrsfs}
\usepackage{slashed}

\usepackage{enumerate} 
\usepackage{xspace} 
\usepackage{xargs}

\newcommand{\mathbbm}{\mathbb}
\DeclareMathOperator{\eval}{ev}

\newcommand{\derp}[2][]{\frac{\partial #1}{\partial #2}}

					\newcommandx*{\ModQHM}[2][1={},2={}]{{M^{#1}_{#2}}}

\newcommand{\N}{\mathbbm{N}}
							\newcommandx*\QHM[2][1={},2={}]{{D^{#1}_{#2}}}
							\newcommandx*\QHMl[2][1={},2={}]{{\mathscr{D}^{#1}_{#2}}}
\newcommand{\R}{\mathbbm{R}}

\newcommand{\Z}{\mathbbm{Z}}

\theoremstyle{definition} 

\DeclareMathOperator{\im}{Im}

\setboolean{notes}{true}
\newcommand{\dog}{Q}
\renewcommand{\sco}{\ensuremath{\mathbb{K}^\infty}}

\begin{document}

\title{\Large{\MakeUppercase{Exact sequences for locally  convex subalgebras of Pimsner algebras with an application to Quantum~Heisenberg~Manifolds}}}
\author{\small{OLIVIER GABRIEL AND MARTIN GRENSING}}
\date{\small{4 October 2010}}
\maketitle
\thispagestyle{empty}
\begin{abstract} \small{We prove six-term exact sequences of Pimsner-Voiculescu type for certain subalgebras of the Cuntz-Pimsner algebras. This sequence may, in particular, be applied to smooth subalgebras of the Quantum Heisenberg Manifolds in order to compute the generators of their cyclic cohomology. Further, our results include the known results for smooth crossed products. Our proof is based on a combination of arguments from the setting of Cuntz-Pimsner algebras and the Toeplitz proof of Bott-periodicity.}
\end{abstract}

\section{Introduction} In classical bivariant $K$-theory, the Pimsner-Voiculescu-six-term-exact-sequence (PMV-sequence)  \cite{MR587369} is one of the main results, allowing to calculate the $K$-theory of crossed products in many  cases. While proved originally for monovariant $K$-theory, it can be proved more simply using $KK$-theory, see \cite{BlackK}. J. Cuntz (\cite{MR750677}) has given a proof that applies, at least partially, to all split exact, homotopy invariant stable functors on the category of $C^*$-algebras with values in the category of abelian groups. His proof is based on the notion of quasihomomorphism, a sort of linearized version of $*$-homomorphism. These techniques were later generalized to the setting of locally convex (or even bornological) algebras (\cite{CuntzMeyer}). The by now classical Toeplitz proof of Bott periodicity (\cite{MR750677}) may then be modified to yield a proof of the six-term exact sequence, and J. Cuntz gives in \cite{MR2240217} an analogous sequence involving smooth crossed products in the locally convex setting. 

Meanwhile, M. Pimsner presents in \cite{MR1426840} a simultaneous generalisation of crossed products and Cuntz-Krieger algebras, and proves a PMV-type sequence, using a certain universal algebra that is an extension of the algebra corresponding to the crossed products by a stabilisation of the base-algebra.

We will show that, under suitable conditions, one may construct such a sequence for locally convex subalgebras of the Pimsner algebras by considering certain subalgebras of the  tensor products of the Toeplitz algebra  already used for the construction of $kk$-theory in \cite{MR2240217}. One of the main results is then the identification of this extended algebra with the base algebra as formulated in Theorem \ref{PMV}. The proof is based on the techniques introduced by Cuntz in the Toeplitz proof of Bott periodicity in the smooth setting. The passage from crossed product to our setting naturally causes difficulties when considering the equivalence between the subalgebra generated by the canonical copy of the base algebra in its corresponding Toeplitz-algebra. It is here that we have to impose further size-restrictions on the algebras we start with.

Changing the stabilisations in the construction of $kk$, or rather by passing to other versions of the smooth Toeplitz extension defined by using Schatten-classes, has some influence on the growth conditions we impose in section \ref{Morita context}, but we will not go into this for the moment.

As an application, we construct six-term-exact-sequences for the Quantum Heisenberg Manifolds (QHM) in $kk$ and cyclic theory, that allow to completely determine the cyclic invariants of the QHM. The $C^*$-QHM can be considered as Pimsner-algebras, and therefore the long exact sequences are available. When dealing with cyclic theory, one has to stay in the smooth category. We show thus that the smooth QHM satisfy our growth conditions, and our techniques thus apply. Consequently, we get Theorem \ref{QHMDIAG} below as a concrete result. 

We leave the application to an explicit calculation of the cyclic periodic cohomology of the QHM to a future article (see \cite{OlPair}).

\section{Preliminaries}
We denote the completed projective tensor product by $\otimes_\pi$. By a locally convex algebra we will mean a complete locally convex vector space that is at the same time a topological algebra, \textsl{i.e.}, the multiplication is continuous. Hence if $A$ is a locally convex algebra, this means that for every continuous seminorm $p$ on $A$ there is a continuous seminorm $q$ on $A$ such that for all $a,b\in A$
$$p(ab)\leq q(a)q(b).$$

If $A$ is a locally convex algebra, we denote by $ZA$ the locally convex algebra of differentiable functions from $[0,1]$ all of whose derivatives vanish at the endpoints. We denote by $ev^A_t:ZA\to A$ the evaluations in $t\in I$.

\begin{Def} Let $\phi_0,\phi_1:A\to B$ be a homomorphism of locally convex algebras. A diffotopy is by definition a homomorphism $\Phi:A\to ZB$ such that $ev^B_i\circ\Phi=\phi_i$ for $i=1,2$.
\end{Def}

\begin{Def} The smooth compact operators are defined as those  compact operators $A\in\mathbb{B}(l^2{\mN})$ such that, if $(a_{i,j})$ is the representation of $A$ with respect to the standard basis, then for all $m,n\in\mN$:
$$\sup_{i,j\in\mN} (1+i)^n(1+j)^m|a_{i,j}|\leq \infty.$$
They are topologized by the increasing family of seminorms $||_{m,n}$ with
$$||A||_{m,n}:=\sum_{i,j} (1+i)^m(1+j)^n|a_{i,j}|.$$
\end{Def}
Note that the elements of $\sco\otimes_\pi B$ are just the matrices with rapidly decreasing coefficients in $B$ (see \cite{CuntzMeyer}, chapter 2, 3.4).
\begin{Def} A functor $H$ on the category of locally convex algebras with values in abelian groups is called 
\begin{itemize}
\item $\sco$-stable if the natural inclusion $\theta:\mc{A}\to \sco\otimes_\pi \mc{A},\; a\mapsto e\otimes a$ obtained by any minimal idempotent $e\in \sco$ induces an isomorphism
$$H(\theta):H(\mc{A})\to H(\sco\otimes_\pi \mc{A}),$$
\item diffotopy invariant, if $H(ev^{\mc{A}}_t)$ is independent of $t\in [0,1]$ for every locally convex algebra $\mc{A}$,
\item split exact if for every short exact sequence 
\[\xymatrix{0\ar[r]&\mc{A}\ar[r]&\mc{B}\ar[r]&\mc{C}\ar[r]&0}\]
of locally convex algebres admitting a split, the sequence
\[\xymatrix{0\ar[r]&H(\mc{A})\ar[r]&H(\mc{B})\ar[r]&H(\mc{C})\ar[r]&0}\]
is exact.
\end{itemize}
\end{Def}
\section{Pimsner Algebras} 
In this section, we give a short introduction to the construction presented by M. Pimsner in \cite{MR1426840}, and collect some necessary details. See also \cite{Katsura02} for generalized version (which we will only use later on). In the sequel, we will simply refer to a $U(1)$-action as a gauge-action.
 
For given right Hilbert module $E$ over a $C^*$-algebra $A$ we denote by $\mathbb{B}_A(E)$ the adjointable operators on $E$. Let $(E,\phi)$ be a \emph{$C^*$-auto-correspondence} of $A$. That is, $E$ is a Hilbert $A$-module, and $\phi:A\to\mathbb{B}_A(E)$ a representation of $A$ on $E$  by adjointable operators. We assume throughout that $\phi$ is injective and $E$ a full $A$-module, and we suppress $\phi$ whenever there is no risk of confusion. The full Fock-space of $E$ is defined as vector space \index{$\mathcal{F}_E$}$\mathcal{F}_E:=\bigoplus_{n=0}^\infty E^{\otimes n}$, where $E^{\otimes n}$ denotes the $n$-fold tensor product $E\otimes_A\cdots\otimes_A E$. There is a $U(1)$-action on $\mathcal{F}_E$ given by $\lambda.(\xi_1\otimes\cdots\otimes\xi_n):=\lambda^n\xi_1\otimes\cdots\otimes \xi_n$, which induces an $\mN$-grading; the projection onto the degree-zero subspace with respect to this grading is denoted $\dog$. The \emph{Pimsner-Toeplitz algebra} $\mc{T}_E$\index{$\mc{T}_E$} is defined as $C^*(T_\xi|\xi\in E)\subseteq \mB(\mc{F}_E)$. Here $T_\xi$\index{$T_\xi$} is the operator  
$$\mathcal{F}_E^n\owns\eta_1\otimes\cdots\otimes\eta_n\to\xi\otimes\eta_1\otimes\cdots\otimes\eta_n \in \mathcal{F}_E^{n+1}.$$
The gauge-action on $\mathcal{F}_E$ induces a gauge-action on $\mathbb{B}_A(E)$ by conjugation; $\mc{T}_E$ is invariant for this action. It thus inherits a $\mZ$-grading, and we denote the subspace of degree $n$ by $\mc{T}_E^n$.

Define the ideal \index{$J(E)$}$J(E)\subseteq \mB_A(\mc{F}_E)$ as the sub-$C^*$-algebra generated by the subalgebras 
$$\mB_A\left(\bigoplus_{n=0}^N\mc{F}^n_E\right),\; N\in\mN.$$
The \emph{(Cuntz--)Pimsner algebra} is by definition \index{$\mathcal{O}_E$}$\mathcal{O}_E:=\mc{T}_E/J(E)$ (more precisely, the image of $\mc{T}$ in the quotient $\mathbb{B}(\mc{F}_E)/J(E)$); the $U(1)$-action preserves the quotient, hence passes to a Gauge action on $\mc{O}_E$ and induces again an $\mZ$-grading. An element $x$ in $\mc{T}_E$ or $\mc{O}_E$ is of degree $n$ iff $\gamma_t(x)={e^{2\pi int}}x$. Note also that the maps $\xi\mapsto T_\xi $ and $S\mapsto S_\xi$ are isometric and linear, where $S$\index{$S$} is $T$ followed by the quotient map. 

We assume now that $E$ is an Hilbert $A$-bimodule with compatible scalar products ${}_A\langle\,\cdot\,|\,\cdot\,\rangle$ and $\langle\,\cdot\,|\,\cdot\,\rangle_A$:
$$\forall\;\zeta,\xi,\eta\in E:\;{}_A\langle\zeta|\xi\rangle \eta=\zeta\langle\xi|\eta\rangle_A.$$
Then note that on vectors of degree greater zero
$$T_\xi T_\zeta^*(\eta_1\otimes\cdots\otimes\eta_n)=\xi\otimes\langle \zeta|\eta_1\rangle_A \eta_2\otimes\cdots\otimes\eta_n={}_A\langle \xi|\eta\rangle\eta_1\otimes\cdots\otimes\eta_n,$$
and $T_\xi T_\zeta^*=0$ on vectors of degree zero. Hence $T_\xi T_\zeta^*=(1-\dog){}_A\langle\xi|\eta\rangle$. Thus in the Pimsner algebra, we have an honest equality $S_\xi S_\zeta^*={}_A\langle\xi|\eta\rangle$. Therefore, denoting the dual $E^*$ of $E$ by $E^{\otimes -1}$,  $\mc{O}_E$ identifies to the Hilbert-module $\oplus_{n\in\mZ} E^{\otimes_A n}$, equipped with the obvious product.

\section{Morita contexts and split exact functors}\label{contexts}
The following definition of Morita context is basically from \cite{MR2240217}.  We add some isomorphisms to the definitions in \cite{MR2240217} in order to make the existence of a Morita-context a weaker condition than being isomorphic.
\begin{Def}\label{contextcond}
Let $A$ and $B$ be locally convex algebras. Then a Morita-context from $A$ to $B$ is given by data $(\phi,D,\psi,\xi_i,\eta_i)$, where $D$ is a locally convex algebra, $\phi:A\to D$, $\psi:B\to D$ are isomorphisms onto subalgebras of $D$, and sequences $\eta_i$, $\xi_i$ in $D$ such that 
\begin{enumerate}
\item $\eta_i \phi(A)\xi_j\in\psi(B)$ for all $i,j$
\item $(\eta_i \phi(a)\xi_j)_{ij}\in\sco\otimes_\pi \psi(B)$
\item\label{convergencecontext} $\sum \xi_i\eta_i \phi(a)=\phi(a)$ for all $a\in A$ (convergence in $\phi(A)$).
\end{enumerate}
\end{Def}
With this definition:
\begin{itemize}
\item If $\phi:A\to B$ is an isomorphism, then we get a Morita context $(\phi^+,B^+,\cdot^+)$ from $A$ to $B$, where $B^+$ is the unitization of $B$, $b\mapsto b^+$ the canonical embedding, and $\phi^+=\phi\circ\cdot^+$
\item In particular, there is now a canonical Morita-context from $A$ to $A$, for any locally convex algebra $A$
\item If $B\subseteq \sco\otimes_\pi C$ is a subalgebra, and we call a corner in $B$ a subalgebra $A\subseteq B$ of the form $\sum_{i=0}^k e_{0i}B\sum_{i=0}^k e_{i0}$, then there is a trivial context from $A$ to $B$
\item If $B$ is row and column-stable ($e_{0i}B,Be_{i0}\subseteq B$ for all $i$), then there is a context from $B$ to $A$, where $B$ is a subalgebra of $\sco\otimes_\pi A$.
\end{itemize}
 
\begin{Def} Let $H$ be a $\sco$-stable functor, then we define $\theta:A\to \sco\otimes_\pi B,\,a\mapsto (\psi^{-1}(\eta_i\phi(a)\xi_j))_{ij}$ and $H(\phi,D,\psi,\xi_i,\eta_i):=H(\theta)$.
\end{Def}
Note that for $a,a'\in A$ we have 
\begin{align*}
(\psi^{-1}(\eta_i\phi(a)\xi_j))_{ij}(\psi^{-1}(\eta_k\phi(a')\xi_l))_{kl}=(\psi^{-1}(\eta_i\phi(a)\sum_{m}\xi_m\eta_m \phi(a')\xi_j))_{ij}
\end{align*}
which equals $\theta(aa')$ by \ref{convergencecontext} in the definition of a Morita context. Hence $\theta$ is indeed a homomorphism.
\begin{Def} A Morita-bicontext between locally convex algebras $A$ and $B$ is given by two Morita-contexts from $A$ to $B$ and $B$ to $A$ respectively, of the form $(\phi, D,\psi,\xi_i^A,\eta_i^A)$ and $(\psi,D,\phi,\xi_i^B,\eta_i^B)$ such that 
\begin{enumerate}
\item $\phi(A)\xi_i^A\xi_j^B\subseteq\phi(A)$, $\eta_i^B\eta_j^A\phi(A)\subseteq \phi(A)$ (left compatibility)
\item $\psi(B)\xi_i^B\xi_j^A\subseteq \psi(B)$, $\eta_i^A\eta_j^B\psi(B)\subseteq \psi(B)$ (right compatibility).
\end{enumerate}
\end{Def}
\begin{Th}\label{contweakeriso} Given two Morita contexts as in the above definition, and a diffotopy invariant, $\sco$-stable functor $H$, we have
$$H(\psi,D,\phi,\xi_i^B,\eta_i^B)\circ H(\phi, D,\psi,\xi_i^A,\eta_i^A)= \id_{H(A)}\text{ if they are left compatible }$$
$$ H(\phi, D,\psi,\xi_i^A,\eta_i^A)\circ H(\psi,D,\phi,\xi_i^B,\eta_i^B)=\id_{H(B)}\text{ if they are right compatible}$$
\end{Th}\newcommand{\potimes}{\otimes_\pi}
\begin{proof}
The proof is an adaptation of the one from \cite{MR2240217} Lemma 7.2. Denote the isomorphism $H(A)\to H(\sco\potimes A)$ given by $\sco$-stability by $\ee_A$. Then more precisely, we have to show that
\begin{align*}
H(\phi,D,\psi,\xi_i^B,\eta_i^B)\circ \ee_B^{-1}\circ H(\phi, D,\psi,\xi_i^A,\eta_i^A)
\end{align*}
is invertible. We suppose left compatibility; then denoting $\theta^A$ and $\theta^B$ the maps $A\to \sco\potimes B$ and $B\to\sco\potimes A$ determined by the two contexts, and multiplying by $\ee_{\sco\potimes A}$  on the left, we see that it suffices to show that the composition $(\sco\otimes\theta_B)\circ\theta_A$ induces an invertible map under $H$. Now this is the map
\begin{equation}\label{compo}a\mapsto (\phi^{-1}(\eta_i^B(\eta_k^A\phi(a)\xi_l^A)\xi_j^B))_{iklj}\end{equation}
and it is diffotopic to the stabilisation as follows. The $L\times L$ matrix with entries $\phi^{-1}(\hat\eta_\alpha(t)\phi(a)\hat\xi_\beta(t))$, $\alpha,\beta\in \mN^2\cup\{0\}$ with
\begin{align*}
\hat\xi_0(t)=\cos(t) 1\hspace{2cm} \hat\xi_{il}(t)=\sin(t)\xi_i^B\xi_l^A\\
\hat\eta_0(t)=\cos(t) 1\hspace{2cm}\hat\eta_{kj}=\sin(t) \eta_k^B\eta_j^A
\end{align*}
yields a diffotopy of the map in equation \ref{compo}.
\end{proof}
\section{Quasihomomorphisms}
We collect some basic properties of quasihomomorphisms as introduced first in \cite{MR733641}, and  later adapted to the locally convex setting. See for example \cite{MR2207702} and \cite{MR2240217} for details.
\begin{Def} A quasihomomorphism between locally convex algebras $A$ and $B$ relative to $\hat B$, denoted $(\alpha,\bar\alpha):A\rrarrow \hat B\trianglerighteq B$, is given by a pair $(\alpha,\bar\alpha)$ of homomorphisms into an algebra $\hat B$ that contains $B$ such that $(\alpha-\bar\alpha)(A)\subseteq B$, $\alpha(A)B\subseteq B$ and $B\alpha(A)\subseteq B$. 
\end{Def}
Note that the definition is symmetric in $\alpha$ and $\bar\alpha$. Also, the definition can be generalized, basically by supposing that $\alpha,\bar\alpha$ have image in the multipliers of $B$, but we will not go into this. That quasihomomorphisms may be used as a tool to prove properties of $K$-theory (on the category of $C^*$-algebras) goes back   to \cite{MR750677}.
\begin{Lem}\label{quasiprops} Let $(\alpha,\bar\alpha):A\rrarrow \hat B\triangleright B$ be a quasihomomorphism, $H$ a split exact functor with values in the category of abelian groups. Then
\begin{enumerate}
\item $(\alpha,\bar\alpha)$ induces a morphism, denoted $H(\alpha,\bar\alpha)$, from $H(A)$ to $H(B)$,
\item for every morphism $\phi: D\to A$, $(\alpha\circ\phi,\bar\alpha\circ\phi)$ is a quasihomomorphism and:
$$H(\alpha,\bar\alpha)\circ H(\phi)=H(\alpha\circ\phi,\bar\alpha\circ\phi)$$
\item if $\phi:\hat B\to \hat C$ is a morphism and $C\trianglelefteq\hat C$ an ideal such that $(\phi\circ\alpha,\phi\circ\bar\alpha):A\rrarrow\hat C\trianglerighteq C$ is a quasihomomorphism, then $H(\phi\circ\alpha,\phi\circ\bar\alpha)=H(\phi)\circ H(\alpha,\bar\alpha)$,
\item \label{Pt:4} If $A$ is generated as a locally convex algebra by a set $X$, $B\trianglelefteq \hat B$ an ideal and $\alpha,\bar\alpha: A\to \hat B$ are two homomorphisms such that $(\alpha-\bar\alpha)(X)\subseteq B$, then $(\alpha,\bar\alpha)$ is a quasihomomorphism $A\rrarrow\hat B\trianglerighteq B$
\item if $\alpha-\bar\alpha$ is a morphism and orthogonal to $\bar\alpha$, then $H(\alpha,\bar\alpha)=H(\alpha-\bar\alpha)$
\item if $H$ is diffotopy invariant, $(\alpha,\bar\alpha):A\rrarrow Z\hat B\trianglerighteq ZB$  a quasihomomorphism, then the homomorphisms $H(ev_t\circ\alpha,ev_t\circ\bar\alpha)$ is independent of $t\in I$. 
\end{enumerate}
\end{Lem}
\begin{proof}
To prove (iv), we let $A$ be generated by $X$, and $x,x'\in X$; then $$\alpha(xx')-\bar\alpha(xx')=(\alpha(x)-\bar\alpha(x'))\alpha(x')+\bar\alpha(x')(\alpha(x')-\bar\alpha(x')).$$ 
The proof of the other statements can be found in the references given above.
\end{proof}
For subsets $X$ and $Y$ in a locally convex algebra $B$, we say $Y$ is $X$-stable if $XY,YX\subseteq Y$.
\begin{Def} Let $X,Y\subseteq B$ be subsets. Then $XYX$ is defined as the smallest closed $X$-stable, hence locally convex, subalgebra of $B$ containing $Y$.
\end{Def}
$XYX$ is obviously independent of the size (but not the topology) of the ambient algebra: If $B\subseteq B'$ as a closed subalgebra and $X,Y\subseteq B$, then it doesn't matter if we intersect $X$-stable subalgebras of $B$ or $B'$. We denote by $X^+$ the set $X$ with a unit adjoint.
\begin{Lem}\label{triviallem} $XYX$ is the locally convex algebra generated by 
$$\left\{\sum_{finite}x_iy_ix_i'|x_i,x_i'\in X^+,\;y_i\in Y\right\}$$
and thus depends only on the locally convex algebras $LC(X)$ and $LC(Y)$ generated in $B$ by $Y$ and $X$, respectively:
$$XYX=X LC(Y)X=LC(X)YLC(X).$$
\end{Lem} 
\begin{Def} For a pair of morphisms $\alpha,\bar\alpha:A\to \hat B$ we call the quasihomomorphism $(\alpha,\bar\alpha):A\rrarrow \hat B\trianglerighteq B$ with $B:=\alpha(A)(\alpha-\bar\alpha)(A)\alpha(A)$ the associated quasihomomorphism. We call a quasihomomorphism minimal, if it is associated to $\alpha,\bar\alpha:A\to \hat B$.
\end{Def}
By the above Lemma, if $X$ generates $A$, then $\alpha(X)(\alpha-\bar\alpha)(X)\alpha(X)=\alpha(A)(\alpha-\bar\alpha)(A)\alpha(A)$, using the notation of Lemma \ref{triviallem} above.
\section{The Smooth Toeplitz algebra and Extension}
We recall the definition and properties of the smooth Toeplitz algebra as introduced in \cite{MR1456322}, compare Satz 6.1 therein:
\begin{Def} The smooth Toeplitz algebra $\stoe$ is defined as the direct sum $\sco\oplus\fn^\infty(\sph^1)$ of locally convex vector spaces with multiplication induced from the inclusion into the $C^*$-Toeplitz algebra.
\end{Def}
It follows from \cite{MR1456322} that $\mc{T}^\infty$ is a nuclear locally convex algebra. We denote $S$ and $\bar S$ the generators of $\stoe$, set $e:=1-S\bar S$, and extract from the proof of Lemma 6.2 in \cite{MR1456322} the following fact:
\begin{Lem}\label{ahomotopy} There exists a unital diffotopy $\phi_t$ such that
$$ \phi_t(S)=S(1- e)\otimes 1+f(t)(e\otimes S)+g(t)Se\otimes 1,$$
where $f,g$ are smooth functions such that $f(0)=0$, $f(1)=1$, $g(0)=1$ and $g(1)=0$ and all derivatives of $f$ and $g$ vanish in $0,1$.
\end{Lem}
Thus $\phi_0:\stoe \to \stoe\otimes_\pi\stoe,\,x\mapsto x\otimes 1$ is the canonical inclusion into the first variable, and $\phi_1$ is the map determined by $S\mapsto \phi_0(S^2\bar S)+(1-S\bar S)\otimes S$.

To unclutter notation, we denote $\phi_0(x):=\hat x$ and $e\otimes x=:\check x$ for all $x\in\stoe$. We also recall that that the smooth Toeplitz algebra fits into an extension 
$$\xymatrix{ 0\ar[r] &\sco\ar[r]& \stoe\ar[r]^-\pi& LC(U)\ar[r]&0}$$
where $LC(U)$ denotes the algebra $\fn^\infty(\mT)$ of smooth functions on the torus. The extension is linearly split by a continuous map $\rho:LC(U)\to \stoe$. The ideal is the  algebra $\sco$ of smooth compact operators, which is isomorphic (as a vector space) to $s\otimes s$, where $s$ is the space of rapidly decreasing sequences, and hence $\sco$ is nuclear. We caution the reader that the smooth compacts are not isomorphism-invariant if we view them as represented inside some $\mK$, where $\mK$ denotes the compacts on a separable, infinite dimensional Hilbert space. They are rapidly decreasing with respect to a \underline{choice} of base.

\section{A Morita-Context for locally convex subalgebras of $\mc{O}_E$}\label{Morita context}

We assume we are given a $C^*$-auto-correspondence $E$ over $A$, and a self-adjoint, locally convex subalgebra $D\subseteq\mathcal{O}_E$. Define $\mc{A}:=D\cap A$ -- considered as a locally convex algebra with the subspace topology. We define $$\mc{E}:=\{\xi\in E|S_\xi\in D\}=S^{-1}(D),$$ which is a linear subspace by linearity of $S$, a right module over the idealizer of $\mc{A}$ and has an $\mc{A}$-valued scalar product, as is easily checked. We assume throughout that $1_A\in \mc{A}$ and (for convenience), that $\mc{A}$ is complete.
\begin{Def}
\label{Def:ToepAlg}
We define $\stoe_D$, the Toeplitz-algebra of $D$, as the closed subalgebra of $D\otimes_\pi \stoe$ (completed projective tensor product) generated by 
$$\mc{A}\otimes 1,\;S(\mc{E})\otimes S,\;\bar S(\mc{E})\otimes \bar S,$$
and by $\iota:\mc{A}\to \stoe_D$ the canonical inclusion. We denote $\toe_D^0$ the algebra generated by the same generators, but without closure.
\end{Def}
\note{ If we don't assume that $\mc{A}$ is complete, we have to deal later on with a functor on the category of noncomplete locally convex algebras. The tensor-products have to be replaced everywhere by noncompleted tensor products, and the subalgebras are then just the algebraically generated ones. Much of what we do still works in this case.}{}

\begin{Def}\label{alpha} We define the algebra $C$ as the closed subalgebra in $D\otimes_\pi\stoe$ generated by $\mc{A}\otimes e$, $S(\mc{E})\otimes Se$ and $\bar S(\mc{E})\otimes e\bar S$; denote $\iota_C:\mc{A}\to C,\; a\mapsto a\otimes e$ the canonical inclusion.  Define $\alpha:=\id_{\stoe_D}$  and $\bar\alpha:=\Ad(1\otimes S)|_{\stoe_D}$. Then  $(\alpha,\bar\alpha):\stoe_D\rrarrow D\otimes_\pi \stoe \trianglerighteq C$ is a quasihomomosphism by design, in fact, the minimal quasihomomorphism associated to $\alpha,\bar\alpha$.
\end{Def}
\note{Note that $\bar\alpha$ is continuous, because multiplication in the locally convex algebra $D\otimes_\pi \mc{T}^s$ is continuous, hence $\Ad(1\otimes S)$ is continuous when restricted to the subalgebra $\stoe_D$.}{} Note that $(\alpha,\bar\alpha)$ is minimal because $C$ is $\alpha(\stoe_D)$-stable.

We denote $\pi_D:\stoe_D\to D\otimes_\pi LC(U)$ the restriction of the projection ($\otimes_\pi$ preserves surjectivity) defined as 
$$ \id_D\otimes \pi:D\otimes_\pi\stoe\to D\otimes_\pi LC(U).$$

In order to simplify upcoming calculations, we extend the definition of $S(\xi)=S_\xi$ as follows. If $I\in \mN^{n}$ is a multiindex, then we denote a tuple $(\xi_{I_1},\ldots,\xi_{I_n})$ of elements of $\mc{E}$ by $\xi_I$  and define $S_{\xi_I}:=S_{\xi_1}\cdots S_{\xi_n}$; we extend $\bar S$ in the same manner. We will also refer to $\xi_I$ as an $\mc{E}$-multivector, and write $|\xi_I|:=|I|$.

Using the relations in $D$ and Lemma \ref{triviallem}, we see that $C=\toe_D^0(1\otimes e)\toe_D^0$ is the locally convex algebra generated by $\sum x_{k,l}\otimes S^ke\bar S^l$, where $x_{k,l}$ is a finite sum of elements of the form $S_{\xi_i}\bar S_{\xi_j}$ with $|\xi_i|=k, |\xi_j|=l$ (and where we identify $\mc{A}$ with elements of degree zero in the obvious way). Thus $C\subseteq D\otimes_\pi \sco\trianglelefteq D\otimes_\pi \stoe$.

We continue to fix a locally convex subalgebra $D\subseteq \mc{O}_E$ such that $D\cap A$ is complete locally convex and unital. We want to analyse the exact sequence 
$$\xymatrix{0\ar[r] &\Ker(\pi_D)\ar[r]& \stoe_D\ar[r]^-{\pi_D}&\Image(\pi_D)\ar[r]&0,}$$
and its homological invariants. To start out, we treat $\stoe_D$.

\begin{Def} Let $\xi_i$ be a sequence of $\mc{E}$-multivectors. Then we define 
$$\Xi_i:=S_{\xi_i}\otimes S^{|\xi_i|}e \text{ and }\; \bar\Xi_i:= \bar S_{\xi_i}\otimes e{\bar S}^{|\xi_i|},$$
the which are elements in $D\otimes_\pi \sco$.
\end{Def}
\begin{Lem}\label{into} $\bar\Xi_i C\Xi_j\subseteq \mc{A}\otimes e$ for all $i,j\in\mN$; for $x=\sum_{k,l=1}^\infty x_{k,l}\otimes S^ke{\bar S}^l\in C$ 
$$p\otimes q(\bar\Xi_ix\Xi_j)= p(\bar S_{\xi_i} x_{|\xi_i|,|\xi_j|}S_{\xi_j}) q(e)$$
for every continuous seminorm $p$ on $\mc{A}$ and $q$ on $\sco$.
\end{Lem}
\begin{proof}
Let $x\in  C$, and assume first that $x=\sum_{k,l} x_{k,l}\otimes S^ke \bar S$ is a finite sum. Then
\begin{align*}
\bar\Xi_ix\Xi_j=&\sum_{k,l} \bar S_{\xi_i}x_{k,l} S_{\xi_j}\otimes e{\bar S}^{|\xi_i|}S^ke{\bar S}^{l} S^{|\xi_j|}e\\
=&\sum_{k,l} \bar S_{\xi_i}x_{k,l} S_{\xi_j}\otimes \delta_{|\xi_i|,k}\delta_{l,|\xi_j|}e\\
=&\bar S_{\xi_i}x_{|\xi_i|,|\xi_j|} S_{\xi_j}\otimes e,
\end{align*}
and the latter is clearly an element of $\mc{A}\otimes e$. By continuity, the result extends to $C$.

For $x$ a finite sum the above yields the equality, and for $x\in C$ with representation $x=\lim_n (x_n=\sum_{k,l}x_{k,l}^n\otimes S^ke{\bar S}^l)$ as a limit of finite sums $x_n$ we obtain
\begin{align*}
p\otimes q(\bar\Xi_ix\Xi_j)=\lim_np\otimes q(\bar S_{\xi_i}x^n_{|\xi_i|,|\xi_j|}S_{\xi_j}\otimes e)=p\otimes q(\bar S_{\xi_i}x_{|\xi_i|,|\xi_j|}S_{\xi_j}\otimes e)
\end{align*}
from which the result follows.
\end{proof}

For simplicity, we assume now that $i\mapsto |\xi_i|$ is of the form $(1,\ldots,1,2,\ldots 2,3\ldots)$, where the same value is obtained on intervals of fixed length $l$, and that the products $S_{\xi_i}\bar S_{\xi_i}$ of same length sum up to one. We will call such sequences $S_{\xi_i},\bar S_{\xi_i}$ admissible.
\begin{Lem}\label{approx} If $S_{\xi_i}$, $\bar S_{\xi_i}$ are admissible and $p(S_{\xi_i}\bar S_{\xi_i})$ of polynomial growth in $i$ for all continuous seminorms $p$, then $(\sum_{i=1}^N \Xi_i\bar\Xi_i)_N$ is an approximate unit for $D\otimes_\pi\sco$.
\end{Lem}
\begin{proof} Let $x\in D\otimes_\pi\sco$, and present $x=\sum_{m,n=1}^\infty x_{m,n}\otimes S^me\bar S^{n}$ with rapidly decreasing coefficients  $x_{m,n}\in D$. Thus for fixed $N$
\begin{align*}
(\sum_{i=0}^N\Xi_i\bar\Xi_i)x=&\sum_{m,n=0}^\infty \sum_{i=0}^N S_{\xi_i}\bar S_{\xi_i}x_{m,n}\otimes S^{|\xi_i|}e\bar S^{|\xi_i|}S^me\bar S^n\\
=&\sum_{m,n=0}^\infty \sum_{i=0}^N (\delta_{|\xi_i|,m} S_{\xi_i}\bar S_{\xi_i})x_{m,n}\otimes S^me\bar S^n.
\end{align*}
Now let $\ee>0$, $p$ a continuous seminorm on $D$ and fix $s,t\in\mN$. Choose a finite subset $F=\mN_{\leq M}\times\mN_{\leq M}\subseteq \mN^2$ such that 
$$\sum_{(m,n)\in\mN^2\setminus F} p(x_{m,n})||S^m e\bar S^n||_{s,t}\leq \ee.$$ We may use the hypothesis to ensure that $F$ has the property that for all $F\subseteq F'$
$$\sum_{(m,n)\in \mN^2\setminus F'} l\sup_{i\leq m}\{p(S_{\xi_i}\bar S_{\xi_i})\} p(x_{m,n})||S^me\bar S^n||_{s,t}\leq \ee.$$
Let $N\geq lM$; then
\begin{align*}
&p\otimes||_{s,t}\left(\sum_{i=0}^NS_{\xi_i}\bar S_{\xi_i} x-x\right)\\
\leq &\sum_{(m,n)\in F}p\otimes||_{s,t}\left( \Big((\sum_{i:i\leq N,|\xi_i|=m}S_{\xi_i}\bar S_{\xi_i}) -1\Big)x_{(m,n)}\otimes S^me\bar S^n\right)\\
&+\sum_{(m,n)\in \mN^2\setminus F}p\otimes||_{s,t}\left( \Big((\sum_{i:i\leq N,|\xi_i|=m}S_{\xi_i}\bar S_{\xi_i}) -1\Big)x_{m,n}\otimes S^me\bar S^n\right)\\
\leq\;\;0\;&+\sum_{(m,n)\in\mN^2\setminus F}p\otimes||_{s,t}(x_{m,n}\otimes S^me\bar S^n)\\
&+\sum_{(m,n)\in\mN^2\setminus F} p\otimes ||_{s,t}(\sum_{i:i\leq N,|\xi_i|=m}S_{\xi_i}\bar S_{\xi_i} x_{m,n}\otimes S^me\bar S^n)\leq 2\ee.\\
\end{align*}
\end{proof}
\begin{Rem} This result is certainly not optimal, but sufficient for our applications. One may  suppose that the $S_{\xi_i}\bar S_{\xi_i}$ is just an approximate unit with an additional property (and thus not ordered by degree) on the appropriate sets, and a similar proof still goes through. Further, as the above condition is rather related to a lower bound for the size of the sequence $S_{\xi_i}$, it would be surprising if polynomial growth was strictly necessary at this point. After all, the $S_\xi\bar S_\xi$ are positive elements in the base-$C^*$-algebra that sum up to one.
\end{Rem}
\begin{Def} We call $D$ tame if it admits admissible sequences $(S_{\xi_i})_i$, $(\bar S_{\xi_i})_i$ of polynomial growth.
\end{Def}
\begin{Th}\label{context} If $D$ is tame, then there is a Morita bi-context from $\mc{A}$ to $C$.
\end{Th}
\begin{proof} 
This follows by combining Lemmas \ref{into} and \ref{approx} above, and using as the inverse context the constant sequences $1\otimes e$.
\end{proof}
Hence by Theorem \ref{contweakeriso} we have the following
\begin{Cor}\label{equivalenceI} If $H$ is a diffotopy invariant $\sco$-stable functor on the category of locally convex algebras, then $$H(C)\sim H(\mc{A}).$$
\end{Cor}

\section{Equivalence of the Toeplitz and base algebra}
We set $\Phi:\stoe_D\to Z(D\otimes \stoe\otimes\stoe)$ as the restriction of $\id_D\otimes (\phi_t)_t$, where $(\phi_t)_t:\stoe\to Z(\stoe\otimes_\pi\stoe)$ denotes the diffotopy from Lemma \ref{ahomotopy}. Further define $\Psi:\stoe_D\to Z( D\otimes_\pi\stoe\otimes_\pi\stoe)$ as the restriction of the map which is constant $\Ad(1\otimes \hat S)$, where $\Ad(1\otimes\hat S)(x):=(1\otimes \hat S)(x)(1\otimes\hat{\bar S})$ for all $x\in D\otimes_\pi\stoe\otimes_\pi\stoe$.

We let $(\Phi,\Psi):\stoe_D\rrarrow Z(D\otimes_\pi\stoe\otimes_\pi\stoe)\trianglerighteq C'$ denote the minimal quasihomomorphism associated to $\Phi,\Psi$, and $C'_t:=ev_t(C')$ the fibre over $t\in [0,1]$ of $C'$. We have the following 
\begin{Lem} $C'_t$ is the closed algebra generated by 
$$S_\xi\otimes (f(t)\check S+g(t) \hat S\hat e),a\otimes \hat e,\bar S_\xi\otimes (f(t)\check{\bar S}+g(t)\hat e \hat{\bar S}),\;\;\; \xi\in\mc{E},a\in\mc{A}.$$
\end{Lem}
\begin{proof}
Using Lemma \ref{triviallem}, it suffices to check that $(\Phi-\Psi)(\stoe_D)$ is $\Psi(\stoe_D)$-stable, and it suffices to do so on the generators. We have got the product of sets:
\begin{align*}\{S_\xi\otimes(f(t)\check S+g(t)\hat S\hat e),a\otimes\hat e\}\{S_{\xi'}\otimes \hat S^2\hat{\bar S},a'\otimes \hat S\hat{\bar S},\bar S_{\xi'}\otimes \hat S{\hat{\bar S}}^2\}=0
\end{align*}
because, in $D\otimes_\pi\stoe\otimes_\pi\stoe$ we may factor the right hand set by a $1\otimes\hat S$ on the left, which is orthogonal to the left factor. Further:
\begin{align*}\bar S_{\xi}\otimes (f(t)\check{\bar S}+g(t)\hat e\hat{\bar S})\cdot S_{\xi}\otimes \hat S^2\hat{\bar S}=0
\end{align*}
by factoring out $1\otimes\hat S^2$. But
\begin{align}\label{whatever}
(\bar S_{\xi}\otimes (f(t)\check{\bar S}+g(t)\hat e\hat{\bar S}))\cdot(a'\otimes \hat S\hat{\bar S})=\bar S_{\xi}a'\otimes g(t)\hat e\hat{\bar S}.
\end{align}
However, 
$$(a\otimes\hat e)(S_\xi\otimes(f(t)\check S+g(t)\hat S\hat e))=aS_{\xi}\otimes f(t)\check S$$
lies in the algebra, thus so does the result in equation \ref{whatever}, because $\mc{A}$ is unital. $\Psi(\stoe_D)$-left stability follows similarly.

\end{proof}
The fibre $C_t$ is thus not constant. It's size over intermediate points is larger than in the endpoints. We denote $\iota_1$ the inclusion of $\stoe_D\to C'_1$ induced by $\stoe\to \stoe\otimes\stoe,\; x\mapsto \check x$.
\begin{Def} We denote $\bar C$ the closed subalgebra generated by the constants $a\otimes k\otimes 1,S_\xi\otimes k\otimes S, \bar S_\xi\otimes k\otimes \bar S$ in $Z(D\otimes_\pi\stoe\otimes_\pi\stoe)$.
\end{Def}
\begin{Th}\label{Theothercontext}
If $D$ is tame, then there is a Morita contex $\Xi_i'$, $\bar \Xi_i'$ from $C'$ to $\bar C\subseteq Z(D\otimes_\pi\stoe\otimes_\pi\stoe)$ that coincides with the context $\Xi$, $\bar\Xi$ in the fibre over zero.
\end{Th}
\begin{proof} We may use the constants $\Xi_i':=\Xi_i\otimes 1$ as a Morita context on $C'$. The conditions are checked as in the previous case for $\Xi_i$ and $\bar \Xi_i$.
\end{proof}
We denote by $\sigma$ the resulting homomorphism
$$\sigma:C'\to  Z(D\otimes_\pi\stoe\otimes_\pi\stoe)\otimes_\pi\sco,\; x\mapsto(\bar\Xi_i' x\Xi_j)_{i,j},$$
and by $\sigma_t$ it's restriction to the fibre over $t$. Recall from Definition \ref{Def:ToepAlg} that $\iota : A \to \stoe_D$ is the inclusion, from Definition \ref{alpha} that $\iota_C :\mc{A}\to C$ denotes the canonical inclusion, and that $(\alpha,\bar\alpha):\stoe_D\rrarrow D\otimes\stoe\trianglerighteq C$ is the quasihomomorphism induced by the identity and $\Ad(1\otimes S)$.
\begin{Th}\label{PMV} Let $H$ be a split exact functor from the category of locally convex algebras to the category of abelian groups. Set $x:=H(\iota)$ and $y:=H(\alpha,\bar\alpha)$.  
\begin{enumerate}
\item\label{firststat} if $H(\iota_C):H(A)\to H(C)$ is right invertible,  $y$ is right invertible;
\item\label{secondstat} if $H$ is diffotopy invariant and $\sco$-stable, $D$ tame,  then $y$ is left invertible.
\end{enumerate}
Hence if $D$ is tame and $H$ a split exact, diffotopy invariant $\sco$-stable functor, $H(\stoe_D)\sim H(\mc{A})$, and the isomorphism is implemented by $\iota$.
\end{Th}
\begin{proof}
The first part is clear by Lemma \ref{quasiprops}. In fact: 
$$H(\alpha,\bar\alpha)\circ H(\iota)=H(\alpha\circ \iota ,\bar\alpha\circ \iota)=H(\iota_C).$$ Hence if $z$ is a right inverse for $H(\iota_C)$, then $H(\iota)\circ z$ is a right inverse of $y$.

In order to prove \ref{secondstat}, we show that there exists a morphism $z'$ such that  $z'\circ y$ is left invertible. For $z'$, we take the inclusion of $D\otimes_\pi \stoe\hookrightarrow D\otimes_\pi\stoe\otimes_\pi\stoe$, in other words, we simply view $(\alpha,\bar\alpha)$ as a quasihomomorphism into a larger algebra, namely as the quasihomomorphism $(\Phi_0,\Psi_0)$.   Because $\phi_t$ is unital,
$$\Phi_1(a\otimes 1)=a\otimes 1\otimes 1=(\Ad(1\otimes \hat S)(a\otimes \hat 1))+(a\otimes \check 1),$$
$$\Phi_1(S_\xi\otimes S)=S_\xi\otimes(\hat S^2\hat{ \bar S}+\check S)=(\Ad(1\otimes \hat S)(S_\xi\otimes S))+(S_\xi\otimes \check S).$$
 Furthermore $\Image(\Ad(1\otimes \hat S))$ and $\Image(\iota_1)$ are orthogonal because $\bar Se=0=eS$.
We now enlarge again the range of the quasihomomorphism by composing with the stabilisation $\sigma$. We thus set $\bar\Phi:=\sigma\circ\Phi$ and $\bar\psi:=\sigma\circ\Psi$. This yields a quasihomomorphism
$$(\bar\Phi,\bar\Psi):\stoe_D\rrarrow Z(D\otimes_\pi\stoe\otimes_\pi\stoe)\otimes\sco\trianglerighteq\bar C\otimes_\pi \sco.$$

Note that $$\bar\Phi_t=\sigma_t\circ\Phi_t.$$
Hence we may apply Lemma \ref{quasiprops}  to deduce:
\begin{align*} H(\bar\Phi_0,\bar\Psi_0)=H(\bar\Phi_1,\bar\Psi_1)=H(\sigma_1)\circ H(\Ad(1\otimes \hat S)\oplus \iota_1,\Ad(1\otimes \hat S))=H(\sigma_1)\circ H(\iota_1).
\end{align*}
Now $\iota_1$ corresponds to the stabilisation map under the obvious isomorphism $$C'\cong\stoe_D\otimes \sco,$$ and therefore $y$ is left invertible.

Now if $D$ is tame, then $y$ has a right inverse because $H(\iota_C)$ has a right inverse -- being invertible by Theorem \ref{context}. As we have shown that $y$ has also a left inverse above, $y$ is invertible (and $H(\iota_C)\circ H(\tau)$ is it's inverse, where $\tau$ is the stabilisation coming from the Morita bicontext $C\to \mc{A}$).  
\newline This proves the theorem.
\end{proof}
The basic idea is to show invertibility by using a Morita context that identifies the noncanonical copy of $\co\otimes A$ sitting at $0$ in $C'$ with a subalgebra of the stabilisation of $\toe_\alpha$, and extends from zero to a context on all fibres.

\section{Determination of the Kernel and Quotient}
We now determine the remaining terms in 
$$\xymatrix{0\ar[r] &\Ker(\pi_D)\ar[r]& \stoe_D\ar[r]^-{\pi_D}&\Image(\pi_D)\ar[r]&0,}.$$
We call the above sequence canonically split, if the split $\id_D\otimes \rho:D\otimes_\pi \fn^\infty(\sph^1)\to D\otimes_\pi\stoe$, where $\rho$ is the split of the smooth Toeplitz extension, restricts to a split of $\pi_D$. We suppose in all this section that the above sequence is canonically split.
\begin{Prop} If $1\otimes e$ is in $\Ker(\pi_D)$, then
$$C=\stoe_D (1\otimes e)\stoe_D=\Ker (\pi_D).$$
\end{Prop}
\begin{proof}  $C=\stoe_D (1\otimes e)\stoe_D$ is   clear by Lemma \ref{triviallem}. $\stoe_D(1\otimes e)\stoe_D\subseteq \Ker(\pi_D)$ is obvious by hypothesis. Because
$$\xymatrix{0\ar[r] &D\otimes_\pi\sco\ar[r]& D\otimes_\pi\stoe\ar[r]^-{1\otimes\pi}& D\otimes_\pi LC(U)\ar[r]&0}$$
stays exact, $\Ker(\pi_D)=\Ker(1\otimes \pi)\cap \stoe_D$.

Now let $x\in \Ker(\pi_D)$, $\ee>0$, and fix a continuous seminorm $p$ on $\stoe_D$, choose a continuous seminorm $q$ such that there is $c_{\rho_D}$ with $p(\rho_D(y))\leq c_{\rho_D}q(y)$, and choose a continuous seminorm $p'$ on $\stoe_D$ such that $q(\pi_D(x'))\leq c_{\pi_D}p'(x')$ for all $x'$. As $x\in \stoe_D$, we may choose a finite sum $x_0=\sum_{I,J} S_{\xi_I}\bar S_{\xi_J}\otimes S^{|I|}\bar S^{|J|}$ with $p(x-x_0),p'(x-x_0)<\ee$. Then $x_0-\rho_D(\pi_D(x_0))\in\Ker(\pi_D)$ and
$$p(x-(x_0-\rho_D(\pi_D(x_0))))\leq p(x-x_0)+p(\rho_D(\pi_D(x-x_0)))\leq (1+c_{\rho_D}c_{\pi_D})\ee.$$ 
Now $x_0-\rho_D(\pi_D(x))$ is easily seen to be an element in the ideal generated by $1\otimes e$ in $\stoe_D$, and thus the other inclusion is proved.
\end{proof}
\begin{Def}\label{gaugesmooth} Let $D\subseteq \mc{O}_E$ be a locally convex algebra. Then we call $D$ \emph{gauge-smooth} if the application
$$\mc{O}_E\to\fn(\mathbb{T})\otimes_\pi\mc{O}_E,\; x\mapsto [t\mapsto\gamma_t(x)]$$
restricts to a continuous homomorphism
$$\tau:D\to \fn^\infty(\mathbb{T})\otimes_\pi D.$$
We say $D$ is generated by $\mc{E}$ if the algebra generated by $S(\mc{E})$, $\bar S(\mc{E})$ and $D\cap A$ is dense in $D$.
\end{Def}
In the setting of metrizable algebras, it suffices that $D$ be invariant  and its elements smooth for the gauge action.
\begin{Prop}\label{image}
Let $D\subseteq \mc{O}_E$ be a locally convex gauge-smooth subalgebra generated by $\mc{E}$. Then $\Image(\pi_{D})$ is isomorphic to $D$.\end{Prop}
\begin{proof}
The existence of the splitting implies that the image of $\pi_{D}$ is closed, hence generated by $S_\xi\otimes U$, $\bar S_\xi\otimes U$ and $\mc{A}\otimes 1$. The (co)restriction of
$$ \id \otimes \eval_1:\mc{O}_E \otimes \fn(\sph^1)  \to \mc{O}_E \otimes \mC \simeq \mc{O}_E,$$
to $\im \pi_D$ defines a continuous morphism of algebras $\im \pi_D \to D$. Checking on generators, one sees that $\tau$ is it's inverse.
\end{proof}

\section{The long exact sequences}
We now use the theory $kk$ to show that there are exact sequences in $kk$ and bivariant cyclic theory that are compatible with the Chern-Character and the boundary maps. In particular, we get, specialising the first algebra to $\mC$, such sequences in $K$-theory and  cyclic theory.

However, to give our first result, we do not really need $kk$-theory, but we obtain a six-term-exact sequence for every half exact $\sco$-stable functor.
\begin{Def} A functor $H$ on the category of locally convex algebras with values in abelian groups is called
 \emph{half-exact} if for every linearly split short exact sequence
\[\xymatrix{ 0\ar[r]&\mc{A}\ar[r]&\mc{B}\ar[r]&\mc{C}\ar[r]&0}\] 
the sequence
\[\xymatrix{ H(\mc{A})\ar[r]&H(\mc{B})\ar[r]&H(\mc{C})}\]
obtained by applying $H$ is exact.
\end{Def}
\begin{Def} If $D\subseteq \mc{O}_E$ is a unital locally convex subalgebra of a Pimsner algebra such that $D\cap A\subseteq D$ is closed, then we say that $D$ is a \emph{smooth Pimsner algebra} if it is tame, gauge-smooth, generated by $\mc{E}$ and canonically split.
\end{Def}
Note that $1\otimes e\in \Ker(\pi_D)$ holds because there is a frame for $\mc{E}$ inside $D$ by hypothesis.
\begin{Th}\label{generalth} Let $D\subseteq \mc{O}_E$ be a smooth Pimsner algebra that. Then for every half-exact, diffotopy invariant and $\sco$-stable functor there is a six-term exact sequence
\[\xymatrix{ H(\mc{A})\ar[r]&H(\mc{A})\ar[r]&H(D)\ar[d]\\
H(\mathscr{S}D)\ar[u]&H(\mathscr{S}\mc{A})\ar[l]&H(\mathscr{S}\mc{A})\ar[l]
}\]
where $\mathscr{S}$ denotes the smooth suspension.
\end{Th}
\begin{proof} First, extend the half exact-functor in the usual fashion to a homology theory $H_n$ -- compare, \textsl{e.g.},  Lemma 4.1.5 \cite{MR2207702}. Observe that $H_n$ has automatically Bott-periodicity, either by carrying over the proof from \cite{MR2240217}, or by universality of $kk$. Apply the long exact sequence for $H$ and Theorem 8.5 to the short exact sequence
$$\xymatrix{0\ar[r] &\Ker(\pi_D)\ar[r]& \stoe_D\ar[r]^-{\pi_D}&\Image(\pi_D)\ar[r]&0.}$$
By Theorem \ref{context}, we may replace the kernel, using the Morita bi-context, with $\mc{A}$. Further $\stoe_D$ may be replaced by $\mc{A}$ using Theorem \ref{PMV} (half exactness implies split exactness). An application of Proposition \ref{image} shows that $\Image(\pi_D)$  can be replaced with $D$.
\end{proof}

We now make use the version of $kk$ for $m$-algebras introduced in \cite{MR1456322}.

\begin{Th} Let $D$ be as above; suppose $D$ is an $m$-algebra, and the Chern-Connes character is a complex isomorphism on $\mc{A}$ (eventually with coefficients), then it is also a complex isomorphism $HP(D)\to K_0(D)$ (with coefficients). In particular, this holds with coefficients $\mC$ for any $D$ with commutative base $\mc{A}$.
\end{Th}
\begin{proof}
This follows by an application of the five Lemma to 
\[\scalebox{0.9}{\xymatrix{ &HP_*(B,\mc{A})\ar[rr]&&HP_*(B,\mc{A})\ar[dl]\\
  &&HP_*(B,D)\ar[ul]\\
kk_*(B,\mc{A})\ar[rr]\ar[uur]&{}&kk_*(B,\mc{A})\ar[dl]\ar[uur]\\
&kk_*(B,D)\ar[ul]\ar[uur]
}}\]
	where the maps between $kk$ and cyclic theory denote the bivariant Chern-Connes character from \cite{MR1456322}, the others come from the long exact sequence.
\end{proof}

\section{Application to Quantum Heisenberg Manifolds}

\begin{Def}[Quantum Heisenberg Manifolds]
Given two real numbers $\mu, \nu$ and an integer $c>0$, we define the Quantum Heisenberg Manifolds (QHMs or $\QHM[c][\mu,\nu]$) as the $C^*$-completion of the algebra
$$ D_0 = \{ F \in C_c(\Z \to C_b(\R \times S^1)) | F(p,x + 1, y) = e(-c p(y - p \nu)) F(p,x,y) \} $$
where $e(x) = e^{2 \pi i x}$, equipped with the multiplication:
\begin{equation}
\label{Eqn:Comp}
 (F_1 \cdot F_2)(x,y,p) = \sum_{q \in \Z} F_1(x,y,q) F_2(x - q 2 \mu,y- q 2 \nu,p-q),
\end{equation}
and where we have now switched to the notational standard, writing the $\mZ$-variable $p$ in the last component.
\end{Def}

In the following, we will write $\QHM $ instead of $\QHM[c][\mu,\nu]$ whenever there is no risk of confusion. One can prove that the involution is:
$$F^*(x,y,p) = \overline{F}(x - 2 p \mu, y - 2 p \nu,p) .$$
It has been proved in \cite{AbadieEE} that $\QHM $ carries an action of $S^1$ -- the \emph{gauge action} -- that turns it into a \emph{generalised crossed product}. Moreover, in \cite{Katsura02}, \textsc{Katsura} proved that all generalised crossed products are (Katsura-)Pimsner algebras.

The grading of $\QHM $ as a Pimsner algebra, is given by the spectral subspaces of the gauge action. Hence:
\begin{equation}
\label{Eqn:Period}
 \QHM ^{(n)} = \{ F(p,x,y) = \delta_{n,p} f(x,y) | f(x + 1, y) = e(-c n (y - n \nu)) f(x,y) \} 
\end{equation}

\begin{Lem}
\label{Lem:MorEqQHM}
Given $(c, \mu, \nu) \in \N^* \times \R^2$ and an integer $n \in \Z$, there is a frame of two elements $\xi_i^n$, $i = 1,2$ in $\QHM ^{(n)}$. In fact
\begin{align*}
(\xi_1^n)^* \xi_1^n + (\xi_2^n)^* \xi_2^n &= 1
&
\xi_1^n (\xi_1^n)^* + \xi_2^n (\xi_2^n)^* &= 1.
\end{align*}
Furthermore, we can choose $\xi_1^n$ and $\xi_2^n$ to be smooth functions.
\end{Lem}

\begin{proof}
On $D^{(0)}$, we get a frame from $\xi_1:=1$, $\xi_2:=0$. For $n\neq 0$, let $U$, $V$ be small neighbourhoods of $[0,\nicefrac{1}{2}]$ and $[\nicefrac{1}{2},1]$, respectively, and $f_1$, $f_2$ a partition of unity subordinate to $U$, $V$. Set $\chi_i = \frac{f_i}{\sqrt{f_1^2 + f_2^2}}$; thus $\chi_1^2 + \chi_2^2 = 1$.

We define $\xi_1^n$ on $U\times S^1$ by setting $\xi_1^n(x,y) = \chi_1(x)$. Choosing $U$ and $V$ small enough $\xi_1$ and, $\xi_2$ may be assumed to vanish with all derivatives on the boundary of a fundamental domain, and using the equation \eqref{Eqn:Period}, $\xi_1^n$ can then be extended to an element of $\QHM ^{(n)}$ . 

A similar process can be applied to $\chi_2$ to obtain $\xi_2^n$. An easy computation on a well chosen fundamental domain yields
\begin{align*}
(\xi_i^n)^* \xi_i^n &= \overline{\chi_i} \chi_i
& 
 \xi_i^n (\xi_i^n)^* &= \chi_i \overline{\chi_i},
\end{align*}
and then $\chi_1^2 + \chi_2^2 = 1$ ensures that
\begin{align*}
(\xi_1^n)^*  \xi_1^n + (\xi_2^n)^*  \xi_2^n &= 1 
&
\xi_1^n (\xi_1^n)^* + \xi_2^n (\xi_2^n)^* &= 1.
\end{align*}
Finally, notice that if $\chi_i$ is smooth, then so is $\xi_i^n$.
\end{proof}

Recall that the \emph{Heisenberg group} $H_1$ is the subgroup of $GL_3(\R)$ of the matrices
\begin{align*} M(r,s,t):=
&\begin{pmatrix}
1 & s & t \\
0 & 1 & r \\
0 & 0 & 1
\end{pmatrix}
&
&\text{ where } r,s,t \in \R
\end{align*}
It has been proved (see \cite{RieffelDefQuant}, section 5) that the Heisenberg group acts on $\QHM $. We use the following expression for the action:
\begin{equation}
\label{Eqn:AlphQHM}
 \alpha_{(r,s,t)}(F)(x,y,p) = e\Big(-p\big(t + c s (x - r -p \mu)\big)\Big) F(x-r,y-s,p).
\end{equation}
The infinitesimal generators $\partial _i$, for $i =1,2,3$ are
\begin{align*}
	\partial_1(F)(x,y,p) =& -\derp[F]{x}(x,y,p) & \partial_3(F)(x,y,p) =& - i 2 \pi p F(x,y,p) \\
\end{align*}
$$ \partial_2(F)(x,y,p) = - \derp[F]{y}(x,y,p) - i 2 \pi c p (x- p \mu) F(x,y,p).$$

Note that $\alpha_{(0,0,t)}$ is just the gauge action. As $M(0,0,t)$ is in the center, the action of $H_1$ preserves the grading and $\partial_1$, $\partial_2$ commute with $\partial_3$. A short calculation shows:
\begin{align}
\label{Eqn:RelComm}
	[\partial_1, \partial_2] &= -c \partial_3 & [\partial_1, \partial_3] &= 0 & [\partial_2, \partial_3] &= 0
\end{align}

\begin{Def} We define the smooth Quantum Heisenberg Manifold $\mc{D}$ as the subalgebra of $H_1$-smooth elements of $D$.
\end{Def}
Note that $\mc{D}$ inherits the $\mN$-grading from $D$, and that $\mc{D}^1$ is actually a bimodule over $\mc{D}^0$.

We recall for the reader's convenience that the algebra of smooth elements for a Lie-group action on a $C^*$-algebra $A$ form a holomorphically closed Fréchet subalgebra $\mc{A}$ of $A$, equipped with the family of seminorms induced by the action of the universal enveloping algebra $U(\mathfrak{g})$ of the Lie algebra.
\begin{Prop}
The smooth Quantum Heisenberg Manifold is a Fréchet algebra and a smooth Pimsner algebra.
\end{Prop}
\begin{proof}
Unitality is clear by definition, and $\mc{D}$ is Fréchet as noted above. $\mc{D}\cap D^0=\fn^\infty(\mT^2)$ with it's natural topology, and is thus closed. It is gauge smooth, because $\alpha_{(0,0,t)}$ is the gauge action itself (compare the remark after Definition \ref{gaugesmooth}). 

It remains to show that $\mc{D}$ is tame. By Lemma \ref{Lem:MorEqQHM}, there is a frame with two elements for every $\mc{D}^n$; denote $\xi_n$ the  sequence of frames ordered by degree. Then it clearly satisfies the conditions on the degree.  Because scaling and adding do not change the equivalence-class of a family of seminorms, we may apply  Poincaré-Birkoff-Witt (see \cite{MR499562}, 17.4) to see that it is enough to show that the seminorms of the form
$$p_{n_3,n_2,n_1}(d):=||\partial_3^{n_3}\partial_2^{n_2}\partial_1^{n_1}(d)||_{\infty}, \; \; d\in\mc{D}$$
yield polynomially increasing sequences on $\xi_n$. We show: There is a constant $C$ that does not depend on $n$ such that 
\begin{equation}\label{inequ}p_{n_3,n_2,n_1}(\xi_n)\leq C(1+n)^{n_3+2n_2}.\end{equation}
Trivializing over $U$ or $V$, we may assume that $\xi_n$ is $\chi_1$ or $\chi_2$ (compare Lemma \ref{Lem:MorEqQHM}) and calculate in the first case:
\begin{align*} ||p_{n_3,n_2,n_1}\xi_n||=||\partial_3^{n_3}\partial_2^{n_2}\partial_1^{n_1}\chi_1||\leq Cn^{n_3}||\partial_2^{n_2}\chi_1^{(n_1)}||\\\leq Cn^{n_3+n_2}||(x+|\mu|n)^{n_2}||
\end{align*}
where $C$ is a universal constant, from whence the above inequality follows in this case. The case of $\chi_2$ is treated similarly.
\end{proof}
Hence, applying Theorem \ref{generalth} to $kk$ and $HP$, we get
\begin{Th}\label{QHMDIAG} There is a commutative diagram
\[\scalebox{0.85}{\xymatrix{K_0(\fn^\infty(\mT^2))\ar[rr]\ar[dr]&&K_0(\fn^\infty(\mT^2))\ar[rr]\ar[d]&&K_0(\mc{D})\ar[ddd]\ar[dl]\\
&HP_0(\fn^\infty(\mT^2))\ar[r]&HP_0(\fn^\infty(\mT^2))\ar[r]&HP_0(\mc{D})\ar[d]\\
&HP_1(\mc{D})\ar[u]&HP_1(\fn^\infty(\mT^2))\ar[l]&HP_1(\fn^\infty(\mT^2))\ar[l]\\
K_0(\mc{D})\ar[uuu]\ar[ur]&&K_1(\fn^\infty(\mT^2))\ar[ll]\ar[u]&&K_1(\fn^\infty(\mT^2))\ar[ul]\ar[ll]\\
}}\]
\end{Th}

\bibliographystyle{alpha}
\bibliography{../../../Fullbib}
\end{document}